\documentclass[11pt]{amsart}

\usepackage{epsfig,amsmath,amsfonts,latexsym, mathtools, soul}
\usepackage[abs]{overpic}
\usepackage{xcolor}
\usepackage{comment}
 \usepackage{palatino}
 \usepackage{hyperref}
 \usepackage{comment}
 \usepackage{subcaption}
 \usepackage{graphicx}

\newtheorem{theorem}{Theorem}

\newtheorem{proposition}[theorem]{Proposition}
\newtheorem{lemma}[theorem]{Lemma}
\newtheorem{corollary}[theorem]{Corollary}

\theoremstyle{definition}
\newtheorem{definition}[theorem]{Definition}

\newtheorem{construction}[theorem]{Construction}

\theoremstyle{remark}
\newtheorem{remark}[theorem]{Remark}

\def\R{\mathbb{R}}

\newcommand{\tgtsn}{(S^{2n+1}, \xi_{st})}

\newcommand{\tgts}{(S^{3}, \xi_{st})}

\newcommand{\mR}{\mathbb{R}}

\newcommand{\rmodz}{\mathbb{R}/\mathbb{Z}}

\numberwithin{theorem}{section}
\theoremstyle{plain}

\begin{document}

\title{Constructions and Isotopies of High-dimensional Legendrian spheres} \author{Agniva Roy}

\address{Department of Mathematics \\ Boston College \\  Chestnut Hill  \\ Massachusetts}

\email{agniva.roy@bc.edu}


\begin{abstract}
We explore the construction of Legendrian spheres in contact manifolds of any dimension. Two constructions involving open books work in any contact manifold, while one introduced by Ekholm works only in $\mathbb{R}^{2n+1}$. We show that these three constructions are isotopic to the Legendrian unknot, thus recovering and generalising a result of Courte and Ekholm, that shows Ekholm's doubling procedure produces the standard Legendrian unknot. \end{abstract}

\maketitle

\section{Introduction}

A contact manifold $(M, \xi)$ is a smooth manifold $M$ equipped with a maximally non-integrable hyperplane field $\xi$. The construction and investigation of contact manifolds has historically been aided by studying distinguished submanifolds that interact suitably with the contact structure. In dimension 3, these submanifolds are either convex hypersurfaces, or Legendrian knots, which are 1-dimensional submanifolds. These in conjunction have helped achieve the classification of tight contact structures on several classes of 3-manifolds e.g. \cite{Honda00a, Honda00b, Giroux01, Giroux00}. A general theme in most of these results can be seen as follows: first understand Legendrian isotopy classes of a family of Legendrian knots in $\tgts$, then understand the contact structures that appear by performing contact surgery on these Legendrians.

In higher dimensions, the first hindrance to carrying out this plan is the shortage of examples of Legendrian spheres. Some constructions of high dimensional Legendrian spheres are explored in \cite{EkholmEtnyreSullivan05a, ekholm2016non, dimitroglou2011knotted, bourgeois2015lagrangian, capovilla-searle_casals}, in $\mR^{2n+1}$ with the standard contact structure. In this article, we describe some general constructions of Legendrian spheres in $(2n+1)$-dimensional contact manifolds from Lagrangian disks in pages of supporting open books. 

\begin{construction}\label{constr:join}
Suppose $(M, \xi)$ is supported by the open book $(B, \nu)$, where $\nu: M-B \to S^1$ is a fibration, and each page is symplectomorphic to $(W, \omega)$. Consider a properly embedded Lagrangian $n$-disk $L$ on the page. Then consider two pages $W, W'$ and two copies of the same Lagrangian, called $L, L'$ on them. The disks can be individually perturbed to give Legendrian disks in $M$. We can further perturb $L$ and $L'$ so they can be smoothly joined to give the union, a Legendrian sphere $L \cup L'$. This now gives a closed Legendrian in $(M, \xi)$. We will call this construction $S_{join}(L)$. Since $L'$ is an isotopic copy of $L$, the notation suppresses $L'$. 
\end{construction}

\begin{construction}\label{constr:stab}
Consider $(M,\xi)$ and $L$ similarly as above. Then, the open book can be stabilised, by modifying the page by attaching a Weinstein $n$-handle along $\partial L$, and then performing a positive Dehn twist along the resulting exact Lagrangian $n$-sphere, obtained by taking the union of $L$ and the core of the handle. This resulting manifold is contactomorphic to $(M,\xi)$. Also, the Lagrangian sphere in the new page can be perturbed to a Legendrian sphere. We will call this Legendrian $S_{stab}(L)$.
\end{construction}

A motivation for looking at these constructions is the search for ``interesting'' embeddings of Legendrian spheres in high dimensions. In high codimension, with the lack of smooth knotting, it is challenging to give general constructions that will produce Legendrian knotted spheres. Since exotic spheres in smooth topology can often be produced by gluing two balls, one can ask whether something similar happens for Lagrangian disks and spheres produced thus.

Our main result is that the two constructions mentioned above do not produce anything ``interesting' -- they give Legendrian spheres that are isotopic to the standard Legendrian unknot. The standard Legendrian unknot is defined to be the Legendrian realisation of the $n$-dimensional unknotted sphere in a Darboux neighbourhood, which is contactomorphic to a Darboux neighbourhood in $\tgtsn$, which is the boundary of an exact Lagrangian disk in $(B^{2n+2}, \omega_{st})$. Its front projection can be inductively constructed by starting from the unknot with maximal Thurston-Bennequin number in $(\mR^3,\xi_{st})$ and successively spinning half of the front projection.

\begin{theorem}
\label{thm:isotopy}
Consider a supporting open book decomposition of a contact manifold $(M, \xi)$. Consider a Lagrangian disk $L$ in the page. Then $S_{join}(L)$ and $S_{stab}(L)$ are both isotopic to the standard Legendrian unknot.
\end{theorem}

\begin{construction}\label{constr:front}
This construction of Legendrian spheres in $(\mR^{2n+1}, \xi_{st})$ was introduced by Ekholm in \cite{ekholm2016non}. Start with a Lagrangian disk $L$ which is cylindrical near its boundary, in the symplectisation of $\mR^{2n-1}$. Embed it in a hypersurface in $\mR^{2n+1}$ transverse to the Reeb flow. Then join the Legendrian lifts of this disk and a reflection of the disk in the same hypersurface to obtain a Legendrian sphere $\Lambda(L,L)$. An example of this construction is given in Figure \ref{fig:Cobordism}. 
\end{construction}

\begin{figure}[htb]{\tiny
\begin{overpic}[width=\textwidth,tics=10] 
{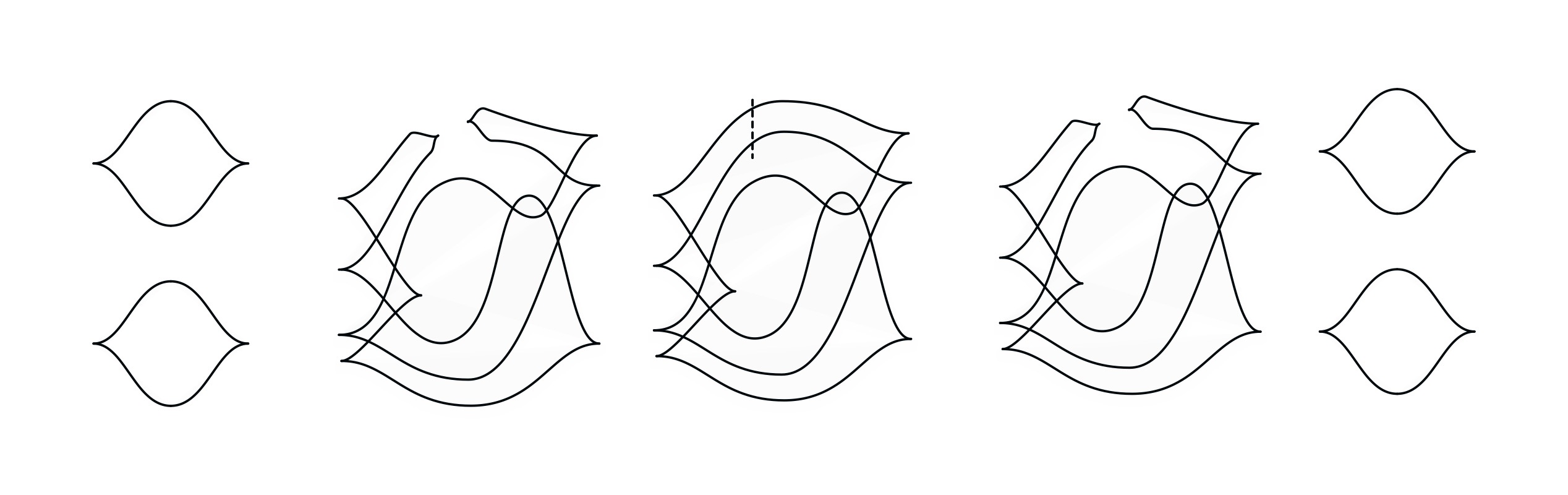}
\end{overpic}}
\caption{Front projection of a Legendrian sphere in $(\mR^5, \xi_st)$, where $L$ is the Lagrangian disk described by a pinch move from the middle knot and going to two Legendrian unknots, which are then capped off. It follows by the main theorem of \cite{courte2017lagrangian} that this is isotopic to the standard unknot.}
\label{fig:Cobordism}
\end{figure}

In \cite{courte2017lagrangian}, Courte-Ekholm show that $\Lambda(L,L)$ is isotopic to the standard Legendrian unknot. The contact manifold $(S^{2n+1}, \xi_{st})$ is supported by an open book where the pages are symplectomorphic to $(D^{2n}, \omega_{st})$, and the monodromy is the identity. The binding is contactomorphic to $(S^{2n-1}, \xi_{st})$. Given a Lagrangian disk $L$ in $(B^{2n}, \omega_{st})$, one can construct all the three Legendrians as mentioned above. We can show that Courte-Ekholm's result is a particular case of Theorem~\ref{thm:isotopy}.

\begin{corollary} (Originally proven in \cite{courte2017lagrangian})
\label{cor:courte}
Given a Lagrangian disk $L$ in $(B^{2n}, \omega_{st})$, $\Lambda(L,L) \subset (S^{2n+1}, \xi_{st})$ is isotopic to the standard Legendrian unknot. 
\end{corollary}

The paper is organised as follows: In Section \ref{sec:contact}, we introduce relevant background and give references for further exposition. In Section \ref{sec:isotopy}, we prove Theorem \ref{thm:isotopy}, namely that the join and stabilisation constructions give spheres isotopic to the standard unknot. Finally we remark on how this generalises Courte-Ekholm's result.

\subsection{Acknowledgements} I would like to thank my advisor John Etnyre for the discussions that led to asking the questions that this paper eventually answers, and also for his patience and guidance in helping me come up with the proof strategies. I am further indebted to his careful reading and comments on various drafts of this paper, and for his assistance in organising the exposition. I thank Sylvain Courte and Tobias Ekholm for answering my questions about their result over email. I am also grateful to Hyunki Min and James Conway for insightful conversations at the beginning of the project. I am further indebted to helpful feedback and suggestions from an anonymous referee. This work was partially supported by NSF grant DMS-1906414.

\section{Background}\label{sec:contact}

In this section we will give the necessary background on contact geometry, Legendrian submanifolds, Legendrian surgery, and open books, to set up the outline for the proofs of the main theorems. The reader is encouraged to consult Otto Van Koert's notes \cite{vanKoert17} for more details.

\subsection{Legendrian submanifolds, front projections, and Legendrian surgery}\label{surgery}

\begin{definition}
Given a contact manifold $(M^{2n+1},\xi)$, an $n$-dimensional submanifold $L$ is called a Legendrian if $T_x(L) \subset \xi_x$ for every $x \in L$. 
\end{definition}

It is well-known that a Legendrian sphere in a contact manifold always has a standard neighbourhood. 

\begin{lemma}\label{lem:std}
If $S$ is a Legendrian $n$-sphere in $(M^{2n+1}, \xi)$, then in any open set containing $S$ there is a neighborhood $N$ with boundary $\partial N = S^n \times S^n$ contactomorphic to an $\epsilon$-neighbourhood $N_{\epsilon}$ of the zero section $Z$ in the 1-jet space of $S^n$, denoted $J^{1}(S^n)$. We call $N$ a standard neighbourhood.
\end{lemma}

The model we will use to describe Legendrian surgery is understood as what is happening on the boundary when a Weinstein handle is attached along the Legendrian sphere. This is called the {\em flat Weinstein model} and is described in Section 3 of \cite{vanKoert17} in more generality, for isotropic surgery along $S^k$ for $k \leq n$. Our description follows the exposition there.

\vfill

{\bf Notation:} To make the notation less cluttered when we talk about $\mR^{2n+2}$, we will write the coordinates $(z_1, w_1, \cdots, z_{n+1}, w_{n+1})$ as $(z,w)$. The symplectic form $\omega_0 = \sum_{i=1}^{n+1} dz_i \wedge dw_i$ will be referred to as $dz \wedge dw$. Similar liberties will be taken with 1-jet space coordinates where $(z, p_1, q_1, \cdots, p_n, q_n)$ will be truncated to $(z, q, p)$, and the contact structure there is $\ker(dz + pdq)$. Products between vectors should be thought of as dot products.

\vfill

Consider the symplectic manifold $(\mR^{2n+2}, \omega_0)$, where the coordinates are $(n+1)$ pairs of $(z,w)$ coordinates, and $\omega_0 = dz \wedge dw$. The vector field $X = 2z \partial_z - w\partial_w$ is Liouville. The set $S_{-1} \coloneqq \{(z,w) \mid |w|^2 = 1\}$ is transverse to $X$ and inherits the contact form $\alpha = 2zdw + wdz$. In $S_{-1}$, the sphere $\{z=0, |w|^2=1\}$ describes a Legendrian sphere. Using $\psi_W : J^{1}(S^n) \to S_{-1}$ given by $(z,q,p) \mapsto (zq+p, q)$, we get a strict contactomorphism between $S_{-1}$ and the standard neighbourhood described in Lemma \ref{lem:std}. Thus $S_{-1}$ can be regarded as the standard neighbourhood of a Legendrian sphere.

Now, Legendrian surgery along a Legendrian $S$ will involve removing a neighbourhood of $S$ identified with $S_{-1}$ and gluing in another contact hypersurface of $(\mR^{2n+2}, \omega_0)$. The contact hypersurface involved in that is called $S_1$ and we describe it here. Define functions $f$ and $g$, described in Figure \ref{fig:handle}, to satisfy the following:
\begin{itemize}
\item $f$ is increasing on $[1-\delta, \infty)$
\item $f(w) = 1$ for $w \in [0, 1- \delta), f(w) = w + \epsilon$ for $w > 1 - \frac{\delta}{2}$
\item $g$ is increasing on $(0, 1 + \delta)$
\item $g(z) = z$ for $z < 1$, $g(z) = 1 + \delta$ for $z > 1 + \delta$
\end{itemize}

\begin{figure}[htb]{\tiny
\begin{overpic}[width=\textwidth,tics=10] 
{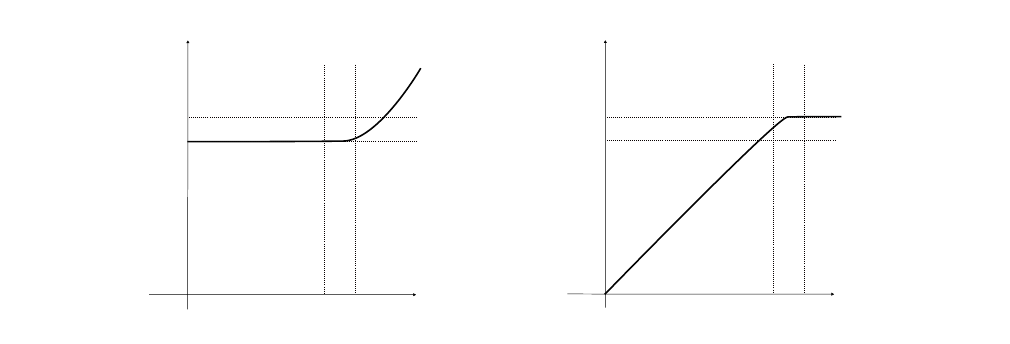}
\put(130,15){$1-\delta$}
\put(160,15){$1$}
\put(190,15){$w$}
\put(80,90){$1$}
\put(68,105){$1+\delta$}
\put(80,130){$f$}
\put(362,15){$1+\delta$}
\put(350,15){$1$}
\put(388,15){$z$}
\put(270,90){$1$}
\put(260,105){$1+\delta$}
\put(270,130){$g$}
\end{overpic}}
\caption{The functions $f$ and $g$ used to describe Legendrian surgery.}
\label{fig:handle}
\end{figure}

\begin{figure}[htb]{\tiny
\begin{overpic}[width=\textwidth,tics=10] 
{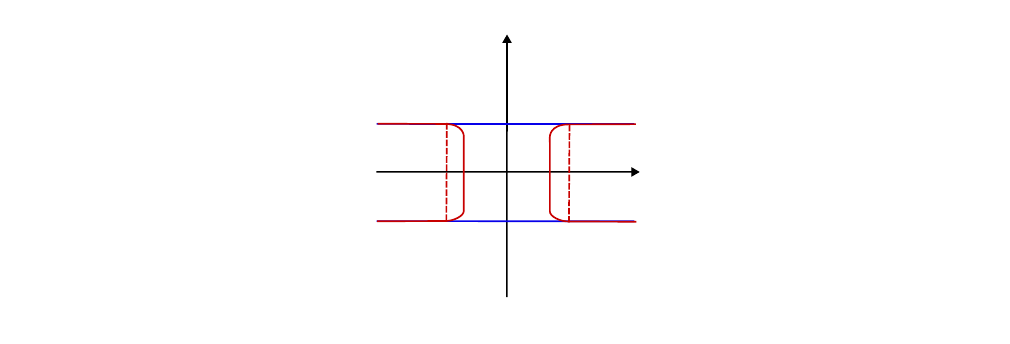}
\put(215,70){$S_{1}$}
\put(245,70){$S_{1}$}
\put(160,100){$S_{-1}$}
\put(160,55){$S_{-1}$}
\put(194,90){$S_1^{st}$}
\put(265,90){$S_1^{st}$}
\put(236,135){$z$}
\put(290,83){$w$}
\end{overpic}}
\caption{The hypersurfaces involved in defining Legendrian surgery.}
\label{fig:surgery}
\end{figure}

Then, define the hypersurface $S_1 \coloneqq \{(z,w)\mid f(w^2) - g(z^2) = 0\}$. As $X$ is transverse to $S_1$, it inherits a contact structure. Then, Legendrian surgery along $S$ is removing $\nu(S) \cong S_{-1}$ and gluing $S_1$ in its place. If $S \subset (M, \xi)$, and there is a symplectic manifold $W$ obtained by attaching a Weinstein handle to part of the symplectisation $(M\times [0,1], d(e^t \alpha))$, along $S$ in $M \times \{1\}$, the Legendrian surgery along $S$ can be understood as the upper boundary of $W$. A schematic of these hypersurfaces is described in Figure~\ref{fig:surgery} -- the solid blue and red lines represent $S_1$ and $S_{-1}$.

\subsection{Weinstein handle attachment}
We briefly review the notion of Weinstein handle attachments here, which we will need to define stabilisation of open books. For more detailed exposition the reader is encouraged to consult \cite{cieliebak2012stein}.

A Weinstein domain is the symplectic analogue of a smooth handlebody. For a $2n$-dimensional domain, Weinstein $k$-handles can have index at most $n$, and are attached along isotropic $(k-1)$-spheres in the convex boundary. Recall that a submanifold $S$ of a contact manifold is called {\em isotropic} if $T_xS \subset \xi_x$ for all $x \in S$.

\begin{definition}
A Weinstein handle of index $k$ is $h^k = D^k \times D^{2n-k}$ with a symplectic structure so that $\partial_{-}h^k = (\partial D^k)\times D^{2n-k}$ is concave, and $\partial_{+}h^k = D^k \times (\partial  D^{2n-k})$ is convex. Moreover, $D^k \times \{0\}$ is isotropic and its intersection with $\partial_{-}h^k$ is an isotropic $S^{k-1}$ in the contact structure induced on $\partial_{-}h^k$. Thus, the attaching sphere of a Weinstein $k$–handle is an isotropic $S^{k-1}$. Given an isotropic sphere $S^{k-1}$
in the convex boundary of a symplectic manifold with a choice of trivialization of its conformal symplectic normal bundle, one can attach a Weinstein k–handle by identifying a neighborhood of the isotropic sphere with $\partial_{-}h^k$.
\end{definition}

A Weinstein handle of index $n$ is called a {\em critical} Weinstein handle, and is attached along a Legendrian sphere. It will be useful for us to understand the local model for attaching a critical Weinstein handle. Consider $\mR^{2n}$ with the symplectic structure $\sum_{i=1}^n dx_i \wedge dy_i$. Now consider $H_{a,b} \coloneqq D_a \times D_b$, where $D_a$ is the disk of radius $a$ in the $x_i$ subspace and $D_b$ the disk of radius $b$ in the $y_i$ subspace. Then, $H_{a,b}$ is a model for the Weinstein $n$-handle $h^n$. The expanding vector field $v = \sum_{i=1}^n -y_idy_i + 2x_idx_i$ induces contact structures on $\partial_{-}h^n = (\partial D_a) \times D_b$ and $\partial_{+}h^n = D_a \times (\partial D_b)$.

\subsection{Contact open books}

The background on contact open books is taken from the lecture notes by Van Koert \cite{vanKoert17}. The reader is referred to the same for more details. 

\begin{definition}\label{def:abstract}
An (abstract) contact open book $(\Sigma, \lambda, \phi)$, or $Open(\Sigma, \phi)$ if we suppress the Liouville form from the notation, consists of a compact exact symplectic manifold $(\Sigma, \lambda)$ and a symplectomorphism $\phi : \Sigma \to \Sigma$ with compact support, i.e., it is identity near $\partial \Sigma$.
\end{definition}

\begin{definition}\label{def:embed}
An (embedded) supporting open book for a contact manifold $(M, \xi)$ is a pair $(\nu, B)$, where $B$ is a codimension-2 submanifold of $M$ with trivial normal bundle, such that 
\begin{itemize}
\item $\nu : (M-B) \to S^1$ is a fiber bundle, such that $\nu$ gives the angular coordinate of the $D^2$-factor of a neighbourhood $B \times D^2$ of $B$, and
\item if $\alpha$ is a contact form for $\xi$, it induces a positive contact structure on $B$ and $d\alpha$ induces a positive symplectic structure on each fiber of $\nu$
\end{itemize}
\end{definition}

The embedded open book constructed from Definition~\ref{def:abstract} is the manifold $\Sigma \times [0,1]/ \sim$, where the equivalence relation $\sim$ identifies all points $(x,t)$ and $(x,t')$ where $x \in \partial \Sigma$, and identifies points $(x,0)$ with $(\phi(x),1)$. For our purpose, we will employ another (but equivalent, up to contact isotopy) way of building a manifold from an abstract open book, where we will have something called the {\em thickened binding}. This construction will work as follows:

\begin{definition}
A manifold constructed from the abstract open book $Open(\Sigma, \phi)$ with {\em thickened binding} is the quotient of the disjoint union of the mapping torus $\Sigma \times [0,1]/ ((x,0) \sim (\phi(x),1))$ and the thickened binding $\partial \Sigma \times D^2$, under the identification $(x, t) \sim (x, 1, t)$, where $x \in \partial \Sigma$, and $\{(x,r,\theta) \mid r \in [0,1], t \in \rmodz\}$ are the coordinates on $\partial \Sigma \times D^2$.
\end{definition}

\begin{figure}[htb]{\tiny
\begin{overpic}[width=\textwidth,tics=10] 
{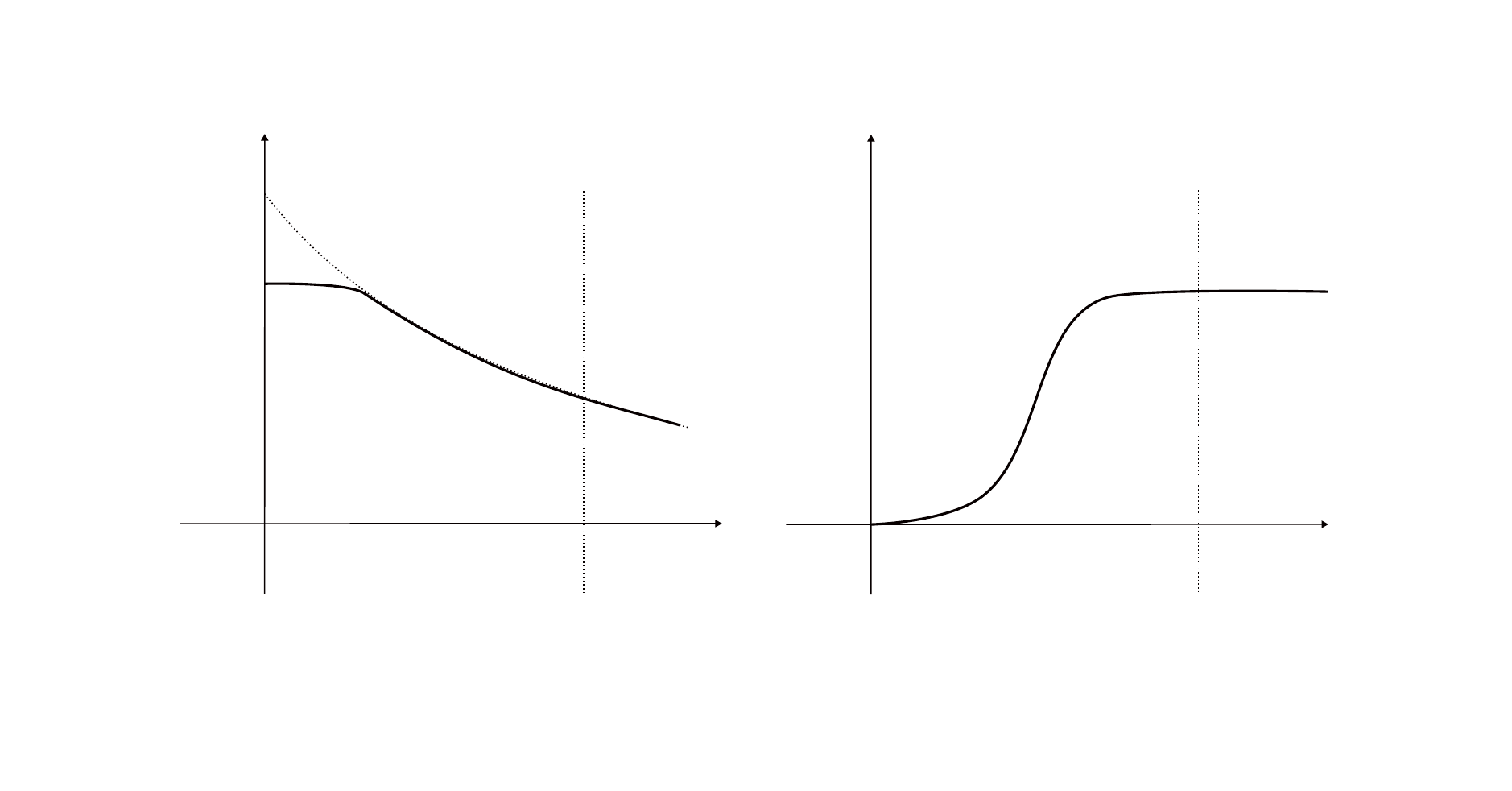}
\put(70,150){$h_1$}
\put(180,70){$\frac{1}{2}$}
\put(257,150){$h_2$}
\put(370,70){$\frac{1}{2}$}
\put(220,70){$r$}
\put(407,70){$r$}
\end{overpic}}
\caption{The functions $h_1$ and $h_2$ used to describe the contact form on an open book near the overlap between the pages and the binding -- $h_1$ has exponential drop-off while $h_2$ is quadratic near 0, and constant near 1.}
\label{fig:obd-forms}
\end{figure}

An open book with thickened binding can be given a compatible contact structure, as shown in Section 2.2 of  \cite{vanKoert17}. In particular, the contact form near the overlap region between the pages and the binding has the form
\[
h_1(r)\lambda|_{\partial \Sigma} + h_2(r)d\theta
\]
where $\lambda|_{\partial \Sigma}$ is the restriction of a Liouville form on $\Sigma$ preserved under the monodromy of the open book, and $(r,\theta)$ are the coordinates on $D^2$ where the binding is $\partial \Sigma \times D^2$. Every contact manifold has a supporting open book decomposition, by work of Giroux. Further, the contact structure supported by an open book is unique upto isotopy, as said by the next theorem, due to Giroux.

\begin{theorem}[Giroux]\label{thm:openbook}
If an open book $(\Sigma, \lambda, \phi)$ supports a contact structure $(M,\xi_1)$, and $\xi_2$ is another contact structure on $M$ supported by an open book whose pages are symplectomorphic to $\Sigma$ and the monodromy is isotopic through symplectomorphisms to $\phi$, then $\xi_1$ and $\xi_2$ are contactomorphic.
\end{theorem}

An abstract open book defines a supporting open book for the corresponding contact manifold. This follows from work of Thurston-Winkelnkemper \cite{ThurstonWinkelnkemper75} and Giroux. The reader can refer to  \cite{vanKoert17} for a proof (originally by Giroux), and more details. For open books, by a {\em page} we refer to $\Sigma$ for abstract open books, and to the closure of a fiber of $\nu$ for embedded ones. In the manifold built from the abstract open book, the equivalence class $[(x,t)]$ for $x \in \partial \Sigma$ is the binding. In the embedded case, $B$ is the binding. In the thickened binding case, $\partial \Sigma \times D^2$ is the binding. In the embedded case, as $M-B$ has the structure of a fibration over $S^1$, it makes sense to talk about the {\em monodromy} of an open book. In the abstract setting, $\phi$ is called the monodromy.

\subsubsection{Generalised Dehn Twist}

Suppose $(W, \omega)$ is a symplectic manifold with an embedded Lagrangian sphere $L \subset W$. A neighbourhood $\nu_W(L)$ is symplectomorphic to a neighbourhood of the zero section of the canonical symplectic structure on $(T^*S^n, d \lambda_{can})$, by the Weinstein neighbourhood theorem. The cotangent bundle of the $n$-sphere $T^*S^n$ can be regarded as a submanifold of $\mR^{2n+2}$ as the set $\{(p,q) \in \mR^{n+1}\times\mR^{n+1} \mid q\cdot q = 1, q \cdot p = 0\}$. In these coordinates, $\lambda_{can} = pdq$. Define an auxiliary map describing the normalised geodesic flow
\[
\sigma_t(q,p) = \begin{pmatrix} \cos t &|p|^{-1}\sin t\\ -|p|\sin t &\cos t \end{pmatrix} \begin{pmatrix} q\\ p\end{pmatrix}
\] 
Then define 
\[ \tau(q,p) = \begin{cases} 
      \sigma_{g_1(|p|)} &  p\neq 0 \\
     -Id & p=0 \\
   \end{cases}
\]

\noindent
$g_1$ is a smooth map as graphed in Figure~\ref{fig:twist}. Since $\tau$ is identity outside a neighbourhood of the Lagrangian $\{p=0\}$, it can be extended to all of $(W, \omega)$ by the identity and defines a symplectomorphism. The map $\tau$ is called the generalised Dehn twist about the Lagrangian sphere $L$.

\begin{figure}[htb]{\tiny
\begin{overpic}[width=0.25\textwidth,tics=10] 
{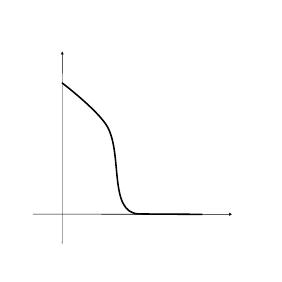}
\put(20,80){$\pi$}
\put(18,95){$g_1$}
\put(93,21){$|p|$}
\end{overpic}}
\caption{The function $g_1$ parametrising a Dehn twist.}
\label{fig:twist}
\end{figure}

\subsubsection{Stabilisation of open books}

Given a contact open book $M =$ Open $(\Sigma^{2n}, \phi)$, suppose $L$ is an embedded Lagrangian $n$-disk in the page $\Sigma$ whose boundary $\partial L$ is a Legendrian sphere in the binding. Consider $\widetilde{\Sigma}$ to be the manifold obtained by attaching a Weinstein $n$-handle to $\Sigma$ along $\partial L$. Then, call $L_S$ the Lagrangian sphere in $\widetilde{\Sigma}$ defined by the union of $L$ and the core of the $n$-handle.

\begin{definition}\label{def:stab}
The contact open book $\tilde{M} \coloneqq$ Open $(\tilde{\Sigma}, \phi \circ \tau_{L_S})$, where $\tau_{L_S}$ is the Dehn twist along $L_S$, is called the {\em stabilisation} of Open$(\Sigma, \phi)$ along $L$.
\end{definition}

The following is a well-known statement due to Giroux. A proof can be found in \cite{vanKoert17}.

\begin{proposition}\label{prop:stabilise}
The stabilisation of a contact open book Open$(\Sigma, \phi)$ along a Lagrangian disk $L$ bounding a Legendrian sphere in $\partial \Sigma$ is contactomorphic to the contact manifold Open$(\Sigma, \phi)$.
\end{proposition}

In Section \ref{sec:isotopy}, we use the following folklore theorem (refer \cite{vanKoert17} for details), that doing Legendrian surgery on a Legendrian sphere that lives on a page is the same as changing the monodromy by a Dehn twist about that sphere.

\begin{theorem}\label{thm:monodromy}
Let Open$(\Sigma, \phi)$ be a contact open book with a Legendrian sphere $L_S$, which is also a Lagrangian sphere in $\Sigma$. Denote the contact manifold obtained from Open$(\Sigma, \phi)$ by Legendrian surgery along $L_S$ by $\widetilde{\text{Open}(\Sigma, \phi)}_{L_S}$. Then, the contact manifolds
\[
\text{Open}(\Sigma, \phi \circ \tau_{L_S}) \simeq \widetilde{\text{Open}(\Sigma, \phi)}_{L_S}
\]
are contactomorphic.
\end{theorem}

\subsection{Generating function Legendrians} A useful way to describe Legendrian submanifolds in 1-jet spaces is by using generating functions. Given $f \in C^{\infty}(Y \times \R^m)$, with $x$ being the $Y$-coordinates and $z$ being the $\R
^m$ coordinates, the critical locus of $f$ is the set 
\[
\Sigma_f \coloneqq \{(x,z) \mid \partial_z(f(x,z)) = 0\}
\]

Then the critical locus embeds into $J^1(Y)$ as a Legendrian submanifold that we will denote $j^1(f)$, where
\[
j^1(f) = \{(f(x,z),x,-\partial_x(f(x,z))) \mid (x,z) \in \Sigma_f\}
\]

It will be useful to understand how to read the coordinates of Legendrians described by generating functions in $J^1(S^n)$, which we will address in the following lemma.

\begin{lemma}\label{lem:function}
Consider $J^1(S^n)$ parametrised as $\{(z,q,p) \mid q^2=1, p \cdot q=0\}$. Given a function $f \in C^{\infty}_{\mR}(S^n \times \R)$, the critical locus $j^1(f)$ defines a Legendrian in $J^1(S^n)$, whose coordinates are given by $(z,q,p)$ such that $z = f(q)$ and $p_i =- \frac{\partial f}{\partial q_i} + (df \cdot q)q_i$, where $df$ is the vector given by $(df)_i = \frac{\partial f}{\partial q_i}$.
\end{lemma}

\begin{proof}
The idea here is simply that when a function $f$ is defined on the coordinates $(q_1, \dots, q_{n+1}) \in \mR^{n+1}$, the vector $df |_{\mR^{n+1}} \coloneqq (\frac{\partial f}{\partial q_1}, \dots, \frac{\partial f}{\partial q_{n+1}})$ belongs to $T^*(\mR^{n+1})$. Thus for $f \in C^{\infty}_{\mR}(S^n)$ but defined on the coordinates $(q_1, \dots q_{n+1})$, $df \in T^*(S^n)$ is given by the projection of $df |_{\mR^{n+1}}$ from  $T^*(\mR^{n+1})$ to $T^*(S^n)$. 
\end{proof}

In particular, if the function $f$ has the property that $\frac{\partial f}{\partial q_i} = 0$ $\forall i \neq n+1$, the $p$ - coordinates, in the cotangent directions, become $p_i = \frac{\partial f}{\partial q_{n+1}}q_{n+1}q_i$ for $i =1,\dots, n$, and $p_{n+1} = \frac{\partial f}{\partial q_{n+1}}(q_{n+1}^2 - 1)$. Thus $p^2 \coloneqq \sum p_i^2 = (\dfrac{\partial f}{\partial q_{n+1}})^2 (1-q_{n+1}^2)$.

\section{The isotopy to the unknot}\label{sec:isotopy}

In this section, we will first rigorously define the constructions $S_{join}$ and $S_{stab}$ in open books supporting $(M, \xi)$. Then, we shall establish an isotopy between them to the standard unknot.

\begin{figure}[htb]{\tiny
\begin{overpic}[width=0.5\textwidth,tics=10] 
{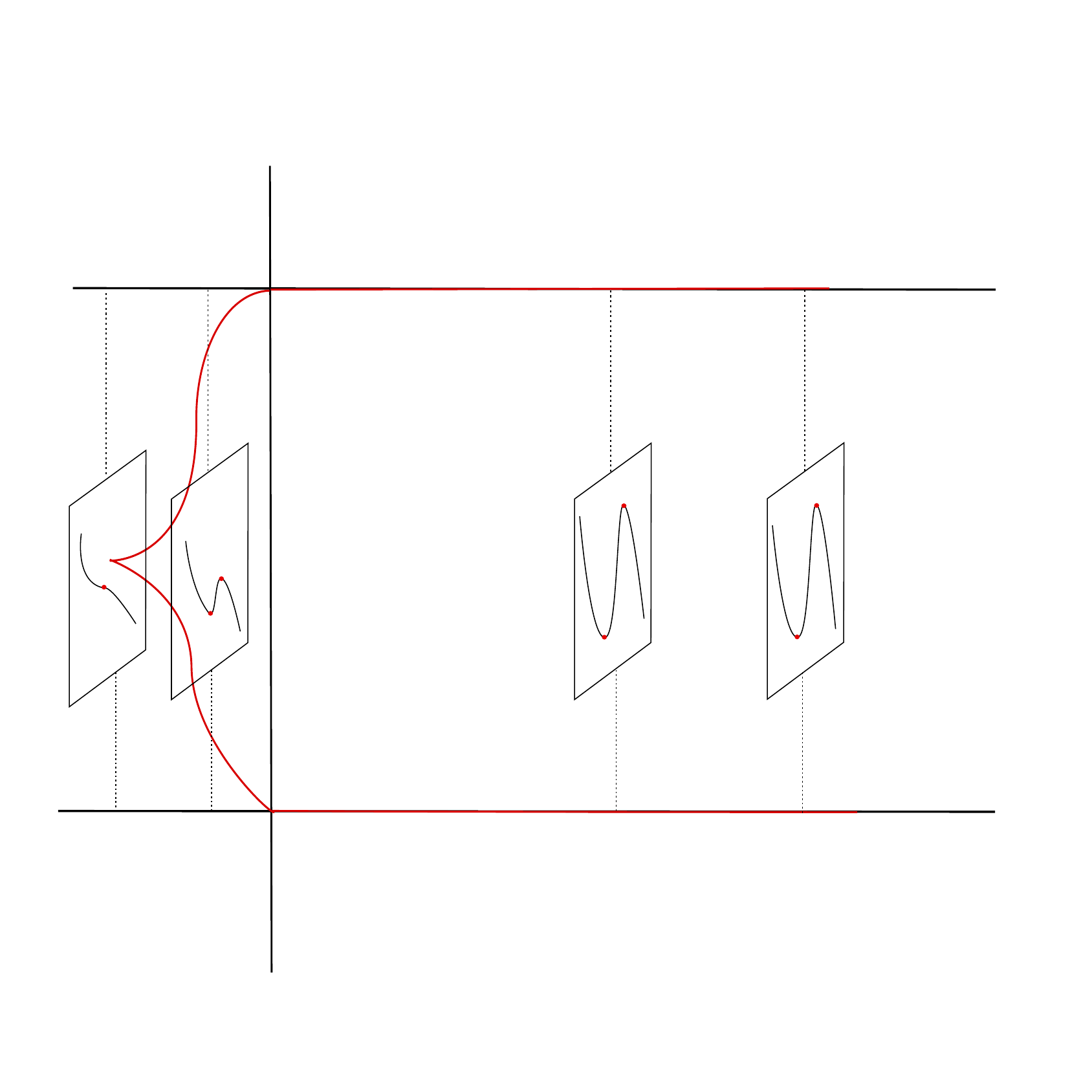}

\end{overpic}}
\caption{A schematic of $S_{join}(L)$ described in a 1-jet neighbourhood of the Legendrian disk $L$, using generating functions whose $q$-slices are shown.}
\label{fig:join}
\end{figure}

\subsection{Parametric definitions of the spheres}\label{ss:spheres}

We are given an open book $(\Sigma, \phi)$ supporting $(M,\xi)$, and a Lagrangian disk $L$ properly embedded in $\Sigma$. By adapting Lemma 4.2 of \cite{vanKoert17}, we can assume that $L$ is a Legendrian in $(M,\xi)$. We can further assume that a neighbourhood of $L$ in $(M,\xi)$ is parametrised as $J^1(D^n) = \{(z,q_1,\dots,q_n,p_1,\dots,p_n)) \mid q \in D^n, p \in \R^n\}$ with the contact form $dz+pdq$, where the $z$ direction gives the fibration direction of the open book. Then $S_{join}$ can be seen to be the union of two disks using generating functions as in Figure~\ref{fig:join}, where each of the disks are given by following one of the critical points in the $q$-slices.

On the other hand, when the open book $(\Sigma, \phi)$ is stabilized along $L$, the sphere $S_{stab}(L)$ can be defined as the union of $L$ and the core of the $n$-handle attached to the page. Parametrically, we can assume that $S_{stab}(L)$ is the 0-section of the $J^1(S^n)$ neighbourhood of this sphere, such that the cotangent directions coincide with the page and the $z$-direction coincides with the fiber direction of the open book.

\subsection{Proof of Theorem~\ref{thm:isotopy}}\label{ss:stab}

We first start with an elementary lemma which is known to experts. We give a proof for the reader's convenience.

\begin{lemma}\label{lem:unknot} Let $S$ be a Legendrian embedding of an $n$-sphere into a connected contact manifold $(M, \xi)$. If $S$ can be described as  two Legendrian $n$-disks $D_1$ and $D_2$ joined at their boundary $\Sigma$, such that $D_2$ can be isotoped rel-boundary to $D_1$, then $S$ is Legendrian isotopic to the standard Legendrian unknot.
\end{lemma}

\begin{proof}
    The above is true if $D_1$ lies in the 0-section of the Darboux neighbourhood $J^1(D^n)$. To prove the lemma we need to show that any two Legendrian embeddings of $D^n$ into $(M,\xi)$ are Legendrian isotopic. Suppose the two embeddings are $h$ and $g$.
    By scaling and isotopy, we can first assume that both $h(D^n)$ and $g(D^n)$ live inside a Darboux neighbourhood and $h(0) = g(0) = p$. By an isotopy of $h$ to $h'$, we can then ensure that the tangent spaces of $h'(D^n)$ and $g(D^n)$ at 0 agree. Since $h'$ and $g$ are both described by smooth mappings of the tangent bundle $D^n$ into $T\xi$, $h'$ can be further isotoped so that they agree at all points. 
\end{proof}

The proof that $S_{join}(L)$ is isotopic to the standard unknot follows directly from its parametric definition above. We can isotope the generating functions to ensure the critical points in the $q$-slices are brought arbitrarily close to each other. Then by applying Lemma~\ref{lem:unknot}, we are done.

We now need to show that in the stabilised open book, $S_{stab}(L)$ is Legendrian isotopic to the standard unknot. There will be two main steps to the isotopy:
\begin{enumerate}
    \item Isotopy through Legendrian surgery
    \item Isotopy through belt-sphere of subcritical handle
\end{enumerate}



\subsubsection{First step of isotopy} \label{sss: step1}

Let $(M, \xi)$ be the contact manifold supported by the open book $\nu$, such that $S_{stab}(L)$ can be defined by stabilizing $\nu$ along the boundary of the Lagrangian disk $L$ in a page of $\nu$. Let $(M',\xi')$ denote the contact manifold supported by the open book $\nu'$, whose pages are obtained by pages of $\nu$ with an $n$-handle attached, and the monodromy is that of $\nu$ extended over the $n$-handle by identity. Let $S_{pre-stab}(L)$ denote the Legendrian lift of the exact Lagrangian sphere formed by $L$ and the core of the $n$-handle. We can assume that $S_{pre-stab}(L)$ lives on a page of $\nu'$. Recall from section~\ref{sec:contact} that the stabilised open book $\nu_{stab}$ supporting $(M,\xi)$ is obtained by a Legendrian surgery along $S_{pre-stab}(L)$. 

By the description of Legendrian surgery in section~\ref{sec:contact}, that means in the final step, we identify a neighbourhood of $S_{pre-stab}(L)$ with $S_{-1} \subset \R^{2n+2}$, and replace $S_{-1}$ with $S_1 \subset \R^{2n+2}$. To make our argument easier, we will use $S_1^{st} = \{(z,w) \mid |z|^2 = 1+\epsilon\}$. Using this, we first consider a ``non-smooth'' version of Legendrian surgery where the set $S_{-1,\epsilon} \coloneqq \{(z,w) \in S_{-1} \mid |z|^2 \leq 1+\epsilon\}$ is replaced by the set $S_{1,\epsilon}^{st} \coloneqq \{(z,w) \in S_1^{st} \mid |w|^2 \leq 1\}$. Refer to Figure~\ref{fig:surgery} for a schematic.

The explicit identification between the standard neighbourhood of Lemma~\ref{lem:std} and $S_{-1}$ is given via the following contactomorphism:

\[
\psi_W: J^1(S^n) \to S_{-1}
\]
\[
(z,q,p) \mapsto (zq+p,q)
\]
So it follows that $\psi_W^{-1}(z_1,w_1) = (z_1.w_1, w_1, z_1 -(z_1.w_1)w_1)$. The explicit identification between the standard neighbourhood of Lemma~\ref{lem:std} and $S_{1}^{st}$ is given via the following contactomorphism:

\[
\psi: J^1(S^n) \to S_{1}^{st}
\]
\[
(z,q,p) \mapsto ((-\sqrt{1+\epsilon}q,p-zq/2)
\]

\begin{figure}[htb]{\tiny
\begin{overpic}[width=\textwidth,tics=10] 
{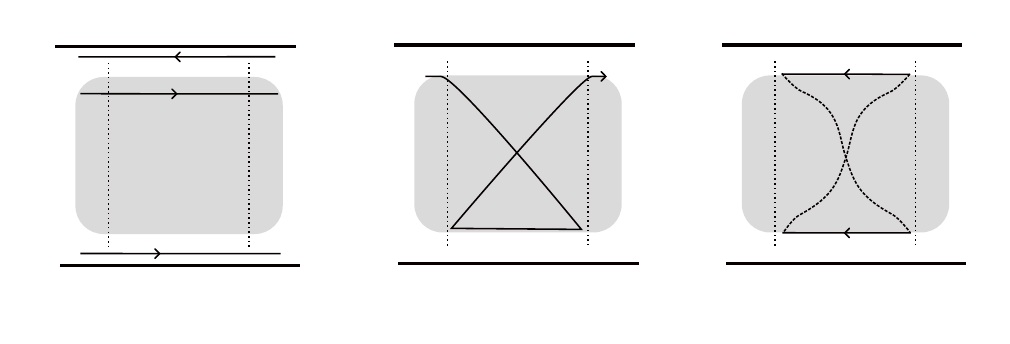}
\put(40,20){Isotopic copies of $S_{stab}(L)$}
\put(70,60){{\color{gray}$N_{surg}$}}
\put(230,60){{\color{gray}$N_{surg}$}}
\put(380,60){{\color{gray}$N_{surg}$}}
\put(200,20){A Reidemeister-1 move}
\put(320,20){After isotoping into $N \setminus N_{surg}$ by $\psi_W^{-1}\psi$}
\put(185,120){$L_{\epsilon,2}$}
\put(243,88){$M_{\epsilon,2,-2}$}
\put(272,50){$C_{\epsilon,-2}$}
\put(375,127){$L'$}
\put(375,85){$M'$}
\put(375,43){$C'$}
\end{overpic}}
\caption{The first step of the isotopy to move $S_{stab}(L)$ from $N$ into $N \setminus N_{surg}$. The vertical directions represents the $z$ coordinate in $J^1(S^n)$, while the horizontal represents $S^n$ in the $q_{n+1}$ direction. The region inside the dotted lines represent $\{q_{n+1} \leq 0\}$. The shaded region represents the part that is replaced during Legendrian surgery, which we denote in Section~\ref{sss: step2} as $N_{surg}$.}
\label{fig:isotopy-1}
\end{figure}

Locally, a neighbourhood of $S_{stab}(L)$ can be identified with $J^1(S^n)$ with coordinates $\{(z,q,p) \mid q^2=1\}$, where $S_{stab}(L)$ is given by $(0,q,0)$. Under the Reeb flow $\partial_z$ in these coordinates, we can see that this 0-section is isotopic to the sphere $(2,q,0)$. Under the map $\psi$, the image of this isotoped sphere is $(-\sqrt{1+\epsilon}q,-q) \subset S^{st}_{1,\epsilon} \cap S_{-1,\epsilon}$. Further under the map $\psi_W^{-1}$, the image of $(-\sqrt{1+\epsilon}q,-q)$ is $\{(\sqrt{1+\epsilon},-q,0)\}$. These observations allow us to visualize a neighbourhood of $S_{stab}(L)$, and its isotopic copies, as described in Figure~\ref{fig:isotopy-1}. We can also see that if we took the core of a 1-jet neighbourhood of $S_{pre-stab}(L)$, isotoped it to $\{c,q,0\}$ for $c>>0$ so that it is in $\psi_W^{-1}(S_{-1} \setminus S_{-1,\epsilon})$, and then performed the Legendrian surgery on the core, this sphere would be isotopic to the core of $\psi^{-1}(S_{1,\epsilon}^{st})$. Which essentially re-affirms our understanding that in $\nu_{stab}$, an isotopic copy of $S_{stab}(L)$ is described by $L \cup \text{core of } n\text{-handle}$ in any page. A schematic of these isotopic copies are shown in the leftmost image in Figure~\ref{fig:isotopy-1}.

Also, for later use, we record that in the standard 1-jet neighbourhood of $S_{stab}$, the image of $S_{-1,\epsilon} \cap S_{1,\epsilon}^{st}$ is given by
\[
\psi^{-1}(S_{-1,\epsilon} \cap S_{1,\epsilon}^{st}) = \{(z,q,p) \mid p^2+z^2/4 = 1\}
\]

As mentioned above, an isotopic copy of $S_{stab}(L)$ in its local 1-jet neighbourhood is parametrised as $\{c,q,0\}$, which is the union of the following pieces:
\begin{itemize}
\item $L_{c} \coloneqq \{(c,q,0) \mid q_{n+1} \geq 0 \}$
\item $C_{c} \coloneqq \{(c,q,0) \mid q_{n+1} \leq 0 \}$
\end{itemize}

Now we can do the first step of the isotopy of $S_{stab}(L)$ to the unknot.

\begin{figure}[htb]{\tiny
\begin{overpic}[width=\textwidth,tics=10] 
{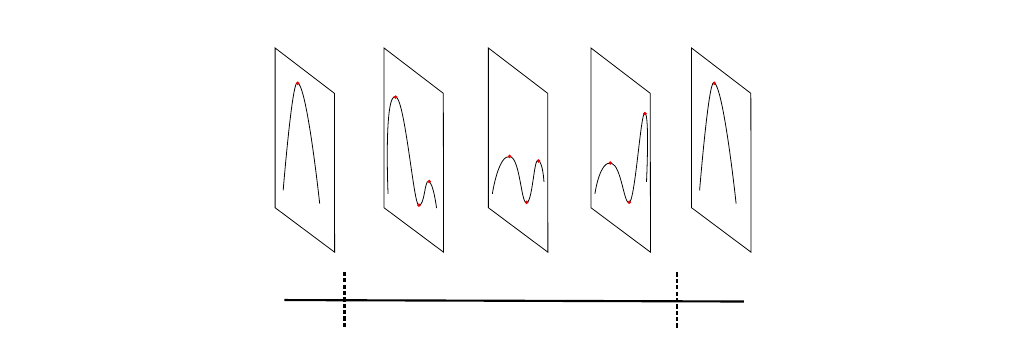}

\end{overpic}}
\caption{The $q_{n+1}$-slices  of the generating functions describing the final step of the first isotopy in Figure~\ref{fig:isotopy-1} using a Legendrian Reidemeister move.}
\label{fig:gen_function}
\end{figure}

Using an isotopy through generating function Legendrians, which corresponds to a Legendrian Reidemeister-1 move, we can isotope it to a sphere as in the middle image of Figure~\ref{fig:isotopy-1}. The generating function description of the sphere in the middle image is given in Figure~\ref{fig:gen_function}. Now, the sphere $S_{stab}(L)$ can be seen as the union of the following pieces:

\begin{itemize}
\item $L_{\epsilon, c} \coloneqq \{(c,q,0) \mid q_{n+1} \geq \epsilon \}$
\item $C_{\epsilon, d} \coloneqq \{(d,q,0) \mid q_{n+1} \leq -\epsilon \}$
\item $M_{\epsilon,c,d} \coloneqq \text{the complement of the above pieces}$
\end{itemize}

If we choose $c=2$ and $d=-2$, the pieces $L_{\epsilon,2}$ and $C_{\epsilon,-2}$ live on $\psi^{-1}(S_{-1,\epsilon} \cap S^{st}_{1,\epsilon})$. Thus, these can be seen in the boundary of $S_{-1} \setminus S_{-1,\epsilon}$, via its identification with $J^1(S^n)$ using $\psi$, as
\[
\psi_W^{-1}\psi(L_{\epsilon,2}) = \{(\sqrt{1+\epsilon},-q,0 \mid q_{n+1} \geq \epsilon\}
\]
\[
\psi_W^{-1}\psi(C_{\epsilon,-2}) = (-\sqrt{1+\epsilon},q,0 \mid q_{n+1} \leq -\epsilon\}
\]

Note that the map $\psi_W^{-1}\psi$, defined on $\psi^{-1}(S_{-1,\epsilon} \cap S^{st}_{1,\epsilon})$ looks like:
\[
(z,q,p) \mapsto (z',q',p') \coloneqq \bigg(\dfrac{z}{2}\sqrt{1+\epsilon},p-\dfrac{zq}{2},(\dfrac{z^2}{4}-1)\sqrt{1+\epsilon}q - \dfrac{z}{2}\sqrt{1+\epsilon}p\bigg)
\]

By choosing generating functions carefully, we can ensure that the $q'$ coordinates on $\psi_W^{-1}\psi(S_{stab}(L))$ satisfy $\{(q')_{n+1} \leq 0\}$. Also, note that the $\sum p^2$ coordinate on $\psi_W^{-1}\psi(S_{-1,\epsilon}\cap S^{st}_{1,\epsilon})$ is given by
\[
p' \cdot p' = ((\dfrac{z^2}{4}-1)\sqrt{1+\epsilon}q - \dfrac{z}{2}\sqrt{1+\epsilon}p)\cdot((\dfrac{z^2}{4}-1)\sqrt{1+\epsilon}q - \dfrac{z}{2}\sqrt{1+\epsilon}p) = (1+\epsilon)p^2
\]

Thus we can perform a Legendrian isotopy on $S_{stab}(L)$ as obtained above, to move it into $\psi_W^{-1}(S_{-1} \setminus S_1)$, by scaling the $z$ and $p$ coordinates by the same constant. The argument currently seems to conflate contactomorphisms (via $\psi$ and $\psi_W$) with isotopies, but that will be argued next.

We now address how to frame our above argument in terms of $S_1$ instead of $S_1^{st}$. The smoothing parameter $\delta$ ensures that the identification between the coordinate systems at $(S_{-1,\epsilon}\cap S^{st}_{1,\epsilon})$, which we wrote using $\psi_W^{-1}\psi$, is smoothed out to be realized by the Hamiltonian flow in $S_1$. The first piece of the argument above, i.e. the Reidemeister-1 move, can be done in the flat part of $S_1$. Then, the the second step of the argument above, instead of using the maps $\psi_W$ and $\psi$, would be achieved using the Hamiltonian flow in $S_1$. The end result of the isotopy however would be the same, and we can consider the three pieces of the parametrized $S_{stab}(L)$ which we will denote
\begin{enumerate}
    \item $L' \coloneqq \psi_W^{-1}\psi(L_{\epsilon,2})$
    \item $M' \coloneqq \psi_W^{-1}\psi(M_{\epsilon, 2, -2})$
    \item $C' \coloneqq \psi_W^{-1}\psi(C_{\epsilon,-2})$
\end{enumerate}

These are represented in the third image of Figure~\ref{fig:isotopy-1}. The dotted lines for $\psi_W^{-1}\psi(M_{\epsilon, 2, -2})$ represents that it belongs to $S_{-1} \setminus S_{1}^{st}$.

\subsubsection{Second step of isotopy}\label{sss: step2}

We can assume that we have chosen the neighbourhoods above carefully so that the $\partial_z$ direction in $J^1(S_{pre-stab}(L))$, or equivalently the part of $J^1(S_{stab}(L))$ identified with $S_{-1}$, agrees with the fibration direction of the open book. We can also assume that $\{q_{n+1} \leq 0\}$ corresponds to the part of the page of $\nu_{stab}$ that comes from the $n$-handle attachment, in the first step of stabilising the open book $\nu$.

The $n$-handle attached to the page is locally $D^n \times D^n$, and the non-trivial monodromy comes from the Dehn twist about the sphere formed by the core and the disk $L$. Over the handle across all pages, can give the fiber direction $S^1$-coordinates from $[-2,2]$ . We can further ensure that in the complement of the part of $J^1(S_{pre-stab}(L))$ that is replaced by $S_1$, the fiber direction agrees with the $z$-coordinate. 

Let us focus on the region of the $\nu_{stab}$ open book of $(M,\xi)$ that includes these $n$-handles of the pages, and the part of the binding attached to them, only. We call it $N$, and it comprises the following pieces, with the fiber direction being given by $z \in [-2,2]$:
\begin{enumerate}
    \item for $z \notin (-\sqrt{1+\epsilon}, \sqrt{1+\epsilon})$, portion of page $\{z=c\}$ is $\{(c,q,p)\mid q_{n+1} \leq 0, p^2 \leq 2\} $. Across all pages, this gives a $I \times D^n \times D^n \subset N$.
    \item for $z \in (-\sqrt{1+\epsilon}, \sqrt{1+\epsilon})$, the portion of the manifold given by $\{(c,q,p) \in N \mid p^2 +\frac{z^2}{4} \leq 1\} $ is identified with $S_{-1,\epsilon}$ and replaced by $S^{st}_{1,\epsilon}$. We will denote this portion as $N_{surg} \subset N$.
    \item To the boundary of the above portions of pages, which is $S^1 \times D^n \times S^{n-1}$, given by $\{(c,q,p) \mid p^2 = 2\}$, the part of the binding that comes inside $N$ is attached, which is $D^2 \times D^n \times S^{n-1}$, with coordinates $((r,\theta),q,p)$ where $(r,\theta)$ belongs to $D^2$ and $p$ belongs to $S^{n-1}$ -- recording the codirections in $U(T^*D^n)$. The attachment is via the following identification 
    \[
    \varphi: S^1 \times I \times D^n \times S^{n-1} \to S^1 \times D^n \times S^{n-1} \times I
    \]
    \[
    (\theta, r, q,p) \mapsto (z,q,p)
    \]
    where the coordinate-wise functions are, by slight abuse of notation on $p$,
    \[
    \varphi(q) = q, \varphi(\theta) = z, \varphi(r,p) = u(r)p
    \]
    where $u(r)$ is a decreasing smooth function on $r$ taking the value $\sqrt{2}$ at $r=1$ and $\sqrt{3}$ at $r=0$.
\end{enumerate}

Inside $N$, if we are near the core of a page, we will use $(z,q,p)$ coordinates and refer to them as ``page coordinates'', and if we are away from the core and in the binding we will use $(r,\theta,q,p)$ coordinates and call them ``binding coordinates''.

The contact structure on $N\setminus N_{surg}$, away from the cores of the pages, is given by the kernel of the standard open book contact form near the binding, which is 
\[
\alpha = h_1(r)pdq + h_2(r)d\theta
\]
with $h_1$ and $h_2$ as in Figure~\ref{fig:obd-forms}.
Recall that in the first step of the isotopy, we isotoped $S_{stab}(L)$ into $N \setminus N_{surg}$. This isotoped $S_{stab}(L)$ is the smooth union of three pieces, namely $L', M',$ and $C'$. On the region of $N \cap \text{ binding}$, denote $I_{surg} \subset S^1$ to be the set defined as $\theta \in I_{surg} \iff (\varphi^{-1}(\theta),q,p) \in N_{surg}$. Consider an orientation reversing identification between $I_{surg}$ and $S^1 \setminus I_{surg}$ that is identity on $\partial I_{surg}$. Say that this sends $\theta \in I_{surg}$ to $\theta_{ref} \in (S^1 \setminus I_{surg})$.


Now we perform an isotopy of $S_{stab}(L)$ inside $(N \setminus N_{surg}) \cap \text{ binding}$. Using binding coordinates, for $0 \leq r < < 1$, the contact form on $(N \setminus N_{surg}) \cap \text{ binding}$ is given by $pdq + r^2d\theta$, where $p \in U(T^*D^n)$ gives the codirections. Carrying forward the coordinates from Section~\ref{sss: step1}, $U(T^*D^n)$ is parametrised as $\{(p_1,\dots,p_{n+1},q_1,\dots,q_{n+1}) \mid \sum q_i^2 = 1, q_{n+1} \leq 0, p\cdot q=0, \sum p_i^2 = 1\}$, thus we have $\sum p_idq_i = 0$. So the contact form in the binding can be written as $-p_{n+1}dq_{n+1} + r^2d\theta$. It follows that for any Legendrian embedding in here with $(r, \theta,q,p)$ coordinates, we must have $\dfrac{dz}{dq_{n+1}} = \dfrac{p_{n+1}}{r^2}$.

By construction, $p_{n+1} \equiv 0$ on $L'$ and $C'$, which live on the ``page'' part of $N \setminus N_{surg}$. Our goal is to isotope $S_{stab}(L)$ inside $N \setminus N_{surg}$ so that $L'$ and $C'$ can be brought arbitrarily close to each other, and then we are done by Lemma~\ref{lem:unknot}. On $L'$, this isotopy is simply by flowing along $\partial_z$. Since we cannot do this isotopy through $N_{surg}$, we will use the binding to do the isotopy along the ``other side'' of the circle, i.e. $S^1 \setminus I_{surg}$.

\begin{figure}[htb]{\tiny
\begin{overpic}[width=\textwidth,tics=10] 
{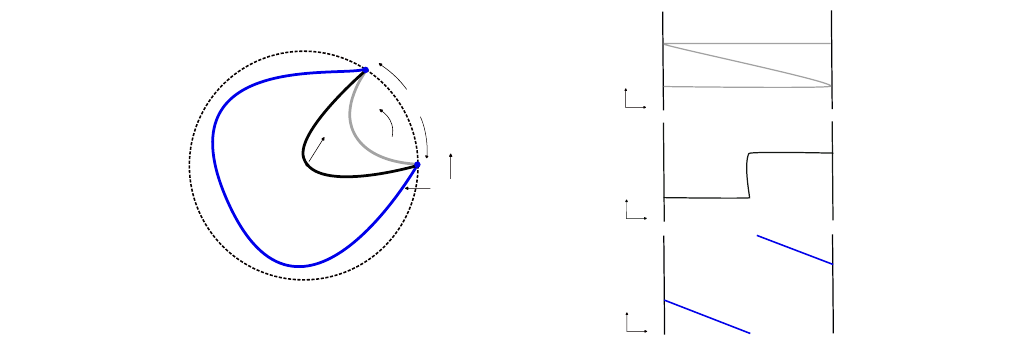}
\put(147,85){$r$}
\put(183,95){$\theta$}
\put(186,110){$I_{surg}$}
\put(195,65){$p$}
\put(210,80){$z$}
\put(285,120){$\theta$}
\put(290,105){$q_{n+1}$}
\put(285,68){$\theta$}
\put(290,53){$q_{n+1}$}
\put(285,16){$\theta$}
\put(290,1){\tiny{$q_{n+1}$}}
\end{overpic}}
\caption{The second step of the isotopy through the binding. We can ``flip'' the $\theta$ coordinates of the portions of the Legendrians where $p_{n+1} \neq 0$ by first making the slope vertical, which takes the Legendrian to the $\{r=0\}$ slice of the binding, and then flowing further to make the $\theta$ coordinate become $\theta_{ref}$. The right side of the above picture shows how this plays out for a component where $\{p_{n+1} < 0\}$, which is represented by the diagonal edge in the top right figure.}
\label{fig:isotopy-2}
\end{figure}

Consider a maximal component of $M'$ where $\{p_{n+1} \neq 0\}$. By an isotopy as shown in Figure~\ref{fig:isotopy-2}, $S_{stab}(L)$ can be isotoped so that the coordinates $(r,\theta,p,q)$ on this component flow to $(r,\theta_{ref},p,q)$, and elsewhere $S_{stab}(L)$ is unchanged. If this is done on all such components, we can now isotope $S_{stab}(L)$ to increase the $r$ coordinate so that it is moved to the ``page'' region of $N \setminus N_{surg}$. Finally, we can isotope $S_{stab}(L)$ in $N \setminus N_{surg}$ to bring $L'$ and $C'$ arbitrarily close to each other, along the lines of how we showed $S_{join}(L)$ was isotopic to the unknot. Our conclusion, similarly to that case, now follows by applying Lemma~\ref{lem:unknot}. This concludes the proof of Theorem~\ref{thm:isotopy}.
 
\begin{remark}
    The proof that $S_{join}(L)$ is the unknot can also be interpreted along the lines of Courte-Ekholm's proof of triviality of doubles \cite{courte2017lagrangian}. Stabilising the open book is a Weinstein cobordism built by attaching cancelling $(n-1)-$ and $n$-handles, and the first step of isotopy in Section~\ref{sss: step1} can be understood as showing that $S_{stab}(L)$ bounds the $(n+1)$-disk left behind on the contact boundary by the belt-sphere of the $n$-handle. The second step of the isotopy in Section~\ref{sss: step2} can be modified to produce a pre-Lagrangian foliation on this disk, as in \cite{courte2017lagrangian}. We are grateful to an anonymous reviewer for suggesting this approach which simplified the proof.
\end{remark}

\subsection{Identifying with $\Lambda(L,L)$ in the $(S^{2n+1}, \xi_{st})$ case}

We will quickly review Ekholm's \cite{ekholm2016non} construction and Courte-Ekholm's proof strategy \cite{courte2017lagrangian}. To start with, one considers a codimension 1 space $W_{\rho} = \{z = (\frac{\rho(x_n)}{\rho'(x_n)})y_n \} \cap \{ 0 < x_n < 1 \}$ in $(\mR^{2n+1}, \xi_{st})$ which is transverse to the Reeb flow. The function $\rho$ can be considered a smoothing of the function $(1 - |x|)$. A Lagrangian disk $L$ with a cylindrical end can be embedded in $W_\rho$ with its cylindrical end approaching $x_n = 0$. Reflecting the $x_n, y_n$ coordinates, another copy of $L$, $L^-$, can be similarly embedded with its cylindrical end approaching that of $L$. Taking the Legendrian lift of these and joining along the ends gives the Legendrian sphere $\Lambda(L,L)$. Deforming the hypersurface $W_\rho$ to $\{z = 0\}$, while staying transverse to $\partial_z$, recovers the construction of $\Lambda(L,L)$ as originally described in \cite{ekholm2016non}. 

To show that this is the unknot, they describe the construction in $(\mR^{2n+1}, ker(dz - \sum_{i=1}^{n-1} y_i dx_i + r_n^2 d \theta_n)$. Then, they modify the contact structure so that two halves of the sphere can be brought close to each other by flowing along $\partial_{\theta_n}$, sketching out a pre-Lagrangian $(n+1)$-disk foliated by Legendrian disks in the process.

\begin{proof}[Proof of Corollary \ref{cor:courte}]By a contactomorphism (refer Example 2.1.10 in \cite{Geiges08}), we can identify a hemisphere of $(S^{2n+1}, \xi_{st})$ with $(\mR^{2n+1}, ker (dz+\sum r_i^2 d\theta_i))$. Under this, the modified $S_{join}(L)$, as in the proof above, which is the Legendrian lift of $L$ in a page, joined with a pushoff, is identified with the Legendrian lift of a disk in $\{z=0\}\cap\{0<x_n<1\}$, joined with a pushoff. (Note that in the proof above we have modified the open book pages so that the Lagrangian in the page is a Legendrian, but for the contactomorphism we want to respect the open book structure given by $\theta_n$, and hence have to use Legendrian lifts.)

Then, by a further sequence of contactomorphisms, we can get to $(\mR^{2n+1}, ker(dz - \sum_{i=1}^{n-1} y_i dx_i + r_n^2 d \theta_n)$, where the Legendrian sphere is still given by the Legendrian lift of a disk in $\{z = 0 \}$ joined with its pushoff. Now deforming the hypersurface $\{z = 0\}\cap\{0<x_n<1\}$ to $W_\rho$, and tracing back through Courte-Ekholm's proof, it is clear that this sphere is in fact isotopic to $\Lambda(L,L)$. Since it is the image under a contactomorphism of the unknot, $\Lambda(L,L)$ is in fact the unknot. This completes the proof of Corollary \ref{cor:courte}.
\end{proof}

\bibliographystyle{amsalpha}
\bibliography{references}

\end{document}